\documentclass[]{article}

\usepackage{amsmath}

\usepackage{amssymb}
\usepackage{amsfonts}
\usepackage{amsthm}
\usepackage{mathtools}

\newtheorem{theorem}{Theorem}[section]
\newtheorem{proposition}[theorem]{Proposition}
\newtheorem{lemma}[theorem]{Lemma}
\newtheorem{corollary}[theorem]{Corollary}
\newtheorem{question}[theorem]{Question}

\theoremstyle{remark}
\newtheorem*{remark}{Remark}
\theoremstyle{definition}
\newtheorem{definition}[theorem]{Definition}

\numberwithin{equation}{section}

\DeclarePairedDelimiter\abs{\lvert}{\rvert}
\DeclarePairedDelimiter\norm{\lVert}{\rVert}
\DeclarePairedDelimiter\alsc{\langle}{\rangle_\alpha}
\DeclarePairedDelimiter\alscn{\langle}{\rangle_{\alpha, n}}

\newcommand{\E}{\mathbb{E}}
\newcommand{\Y}{\mathbb{Y}}
\renewcommand{\P}{\mathbb{P}}
\newcommand{\J}[1]{J_{#1}^{(\alpha)}}
\newcommand{\A}{\mathcal{A}^\alpha_\lambda(n)}
\newcommand{\R}{\mathbb{R}}
\newcommand{\N}{\mathbb{N}}
\newcommand{\Z}{\mathbb{Z}}
\newcommand{\C}{\mathbb{C}}
\newcommand{\T}{\mathbb{T}}
\newcommand{\I}{\mathbb{I}}

\newcommand{\normHT}[2]{\norm{#2}_{\dot{H}^{#1}(\T)}}
\newcommand{\normHR}[2]{\norm{#2}_{\dot{H}^{#1}(\R)}}

\newcommand{\cbee}[1]{\E_{#1}^n}
\DeclareMathOperator{\Conf}{Conf}

\newcommand\blfootnote[1]{%
	\begingroup
	\renewcommand\thefootnote{}\footnote{#1}%
	\addtocounter{footnote}{-1}%
	\endgroup
}
\title{Gessel-Type Expansion for the Circular $\beta$-Ensemble and Central Limit Theorem for the Sine-$\beta$ Process for $\beta\le 2$}
\author{Sergei M. Gorbunov${}^*$
\blfootnote{Institute for System Programming of the Russian Academy of Sciences, Moscow, Russia}
\blfootnote{Moscow Institute of Physics and Technology, Dolgoprudny, Moscow Region, Russia}
\blfootnote{Steklov Mathematical Institute of Russian Academy of Sciences, Moscow, Russia}}
\date{}
\begin{document}
\maketitle
\begin{abstract}	
We obtain a Gessel-type expansion in Jack polynomials for the expectations of multiplicative functionals in the circular $\beta$-ensemble. As a consequence, we establish a Szegő-type limit theorem for all $H^{1/2}(\T)$ functions when $\beta \le 2$, together with an explicit rate of convergence for functions from $H^1(\T)$. The estimate is stable under the scaling limit to the sine-$\beta$ process and yields a Soshnikov-type central limit theorem for the sine-$\beta$ process in the full $H^{1/2}(\R)$ class.\end{abstract}
\section{Introduction}
\subsection{The Gessel Theorem for the circular $\beta$-ensemble}
The circular $\beta$-ensemble is the probability measure on $n$-point configurations on the unit circle defined by
\[
d\P_{\beta}^n(\theta_1, \ldots, \theta_n) = Z_{n, \beta}^{-1}\prod_{1\le l< m\le n}\abs{e^{i\theta_l}-e^{i\theta_m}}^\beta\prod_{k=1}^nd\theta_k, \quad \theta_j\in(-\pi, \pi),
\]
where $Z_{n,\beta}$ is a normalization constant.

For $\beta=2$, this ensemble coincides with the circular unitary ensemble, that is, the eigenvalue distribution of a Haar-distributed unitary matrix.

A classical result of Ira Gessel establishes a representation of multiplicative functionals in the circular unitary ensemble in terms of symmetric functions. Specifically, expectations of multiplicative functionals admit an expansion in Schur polynomials $\{s_\lambda\}_\lambda\subset \Lambda$ indexed by partitions $\lambda$. We refer to Subsection~\ref{sec:jack} for notation and background on symmetric functions.

We say that a function $f$ on the unit circle is $1/2$-Sobolev regular if
\[
\sum_{j\in\Z}\abs{j}\abs{\hat{f}_j}^2<+\infty, \quad f(\theta)=\sum_{j\in\Z} \hat{f}_j e^{ij\theta}.
\]

\begin{theorem}[{The Gessel Theorem \cite[Theorem 16]{G_90}, \cite[II.2]{TW_01}}]\label{introduction:gessel}
If $f$ is $1/2$-Sobolev regular, then
\[
e^{-n\hat{f}_0}\E_2^n\left\{\prod_{j=1}^ne^{f(\theta_j)}\right\} = \sum_{\lambda:l(\lambda)\le n}s_\lambda(\rho_+)s_\lambda(\rho_-),
\]
where the algebra homomorphisms $\rho_\pm:\Lambda\to\C$ are determined by their values on the Newton power sums
\[
p_j(\rho_\pm) = j\hat{f}_{\pm j}, \quad p_j(x)=\sum_{l\in\N}x_l^j\in \Lambda.
\]
\end{theorem}

This identity may be deduced from the Szegő–Heine theorem \cite[Theorem 1.5.13]{S_05}, which expresses the expectation as a Toeplitz determinant, combined with the Cauchy–Binet formula and the Jacobi–Trudi identity. See Tracy and Widom \cite[II.2]{TW_01} for this derivation.

Following Tracy and Widom \cite[II.3]{TW_01}, Borodin and Okounkov \cite[Section 3]{BO_00}, the Szegő limit theorem follows immediately from Theorem~\ref{introduction:gessel}. Applying $\rho_\pm$ to the Cauchy identity
\[
\sum_\lambda s_\lambda(x)s_\lambda(y) = \exp\left(\sum_{k\ge 1}\frac{p_k(x)p_k(y)}{k}\right),
\]
one obtains
\[
e^{-n\hat{f}_0}\E_2^n\left\{\prod_{j=1}^ne^{f(\theta_j)}\right\}  \to e^{\sum_{k\ge 1}k\hat{f}_k\hat{f}_{-k}}, \quad \text{ as }n\to\infty.
\]

The above argument relies crucially on the determinantal structure of the circular unitary ensemble. For general $\beta$, no such structure is available. However, Schur polynomials admit a one-parameter deformation: the Jack polynomials $J_\lambda^{(\alpha)}$ introduced by Henry Jack \cite{J_70}. These are indexed by partitions $\lambda$ and depend on a positive parameter $\alpha$. See Theorem \ref{Jack:def} for their definition due to Macdonald \cite{M_79}, and Section \ref{sec:jack} for their properties.

Macdonald showed \cite[(10.36)]{M_79} that Jack polynomials are orthogonal with respect to the scalar product induced by the circular $\beta$-ensemble with $\beta=2/\alpha$:
\begin{equation}\label{intro_eq:jack_orthogonality}
  \E_{2/\alpha}^n\J{\lambda}(e^{i\theta_1}, \ldots, e^{i\theta_n})\overline{\J{\mu}(e^{i\theta_1}, \ldots, e^{i\theta_n})} = 0, \quad \text{ for }\lambda\ne \mu.
\end{equation}
For $\alpha=1$, the Jack polynomials reduce to Schur polynomials and the orthogonality follows from the Andr\'eief identity \cite{A_83}.

Moreover, the Jack polynomials satisfy the Cauchy identity
\begin{equation}\label{intro_eq:cauchy}
\sum_{\lambda\in\Y}\frac{\J{\lambda}(x)\J{\lambda}(y)}{\alsc{\J{\lambda}, \J{\lambda}}} = \exp\left(\frac{1}{\alpha}\sum_{j=1}^\infty \frac{p_j(x) p_j(y)}{j}\right).
\end{equation}
These two statements --- the orthogonality and the Cauchy identity --- are sufficient to extend Gessel’s expansion to arbitrary $\beta$.

\begin{theorem}\label{intro:gessel}
  Let $f$ be $1/2$-Sobolev regular and assume $\alpha\ge1$. Then
  \[
    \E_{2/\alpha}^n\left\{\prod_{j=1}^ne^{f(\theta_j)}\right\} = \sum_{\lambda:l(\lambda)\le n} \frac{\J{\lambda}(\rho_-)\J{\lambda}(\rho_+)}{\alsc{\J{\lambda}, \J{\lambda}}}\A,
\]
where
  \[
  p_j(\rho_\pm) = \alpha j\hat{f}_{\pm j},
  \]
  and
  \[
  \mathcal{A}_\lambda^\alpha(n) = \prod_{(i, j)\in \lambda}\left(1 - \frac{\alpha - 1}{n+j\alpha - i}\right).
  \]
  The corresponding infinite Jack-polynomial series converges absolutely.
\end{theorem}
\begin{remark}
The assumption $\alpha\ge1$ guarantees absolute convergence of the infinite Jack series; the identity remains valid formally as a power series in the Fourier coefficients of $f$.
\end{remark}

Theorem \ref{intro:gessel} is proved in Subsection \ref{ssec:cbe_th_proof}.

\subsection{Szeg\H{o}'s Limit Theorem for the circular $\beta$-ensemble}

Note that $\A\to 1$ as $n\to\infty$. Combining this convergence with the action of the homomorphisms $\rho_\pm$ on the Cauchy identity~\eqref{intro_eq:cauchy} for the Jack polynomials yields the following extension of the classical Szeg\H{o} theorem.
\begin{corollary}\label{res:cbe_th}
Let $f$ be a $1/2$-Sobolev regular function on the unit circle and assume $\beta \le 2$.
\begin{itemize}
\item The Laplace transform of the centered additive functional converges to that of a Gaussian random variable:
\[
e^{-n\hat{f}_0}\cbee{\beta} \left\{\prod_{j=1}^n e^{f(\theta_j)}\right\} \to e^{\frac{2}{\beta}\sum_{k= 1}^\infty k\hat{f}_k\hat{f}_{-k}}, \quad \text{ as }n\to\infty.
\]
\item If $f$ is real-valued, the centered additive functional satisfies a subgaussian exponential moment bound:
\begin{equation}\label{intro_eq:subgauss}
e^{-n\hat{f}_0}\cbee{\beta} \left\{\prod_{j=1}^n e^{f(\theta_j)}\right\}\le e^{\frac{2}{\beta}\sum_{k\ge 1}k\hat{f}_k\hat{f}_{-k}}.
\end{equation}
\end{itemize}
\end{corollary}

Corollary \ref{res:cbe_th} is proved in Subsection \ref{ssec:cbe_th_proof}.

Under a stronger regularity assumption $f\in C^{1+\varepsilon}$, $\varepsilon>0$, Corollary~\ref{res:cbe_th} was established by Johansson \cite[Remark~2.6]{J_98} for all $\beta>0$. His argument is based on the loop (fundamental) equation; see \cite[Section~3.1]{J_98} and \cite[Lemma~2.1]{L_21}. This method was subsequently developed by Lambert \cite{L_21}, who proved convergence at certain mesoscopic scales under a regularity assumption $f\in C^{3+\varepsilon}$, $\varepsilon>0$; see \cite[Theorem~1.1]{L_21}.

Below we state a quantitative version of Corollary~\ref{res:cbe_th} (Theorem~\ref{res:cbe_est}). The estimate holds in the microscopic regime and, therefore, implies convergence at arbitrarily small mesoscopic scales. For $\beta=2$, the convergence at arbitrarily small mesoscopic scales for Schwartz test functions was previously obtained by Soshnikov \cite[Theorem~1]{S_00}.

We emphasize that $1/2$-Sobolev regularity is necessary for the limiting expression in Corollary~\ref{res:cbe_th} to exist. In \cite[Remark~1.3]{L_21}, Lambert conjectured that this condition is sufficient if and only if $\beta \le 2$. Our result confirms the sufficiency for $\beta \le 2$.

The Jack-polynomial approach is not new. This paper is inspired by the work of Jiang and Matsumoto \cite{JM_15}. 
Jiang and Matsumoto studied joint moments of the Newton power sums $\{p_j\}_{j\in\mathbb N}$ in the circular $\beta$-ensemble. For general $\beta>0$ and a polynomial function $f$, they proved convergence of the additive functional $\sum_{j=1}^n f(\theta_j)$ \cite[Theorem 2]{JM_15}. Their results are particularly precise in the cases $\beta=1$ and $\beta=4$.

For $\beta=1$, Jiang and Matsumoto established the convergence for all $1/2$-Sobolev regular functions \cite[Theorem 3]{JM_15}. For $\beta=4$, they obtained \cite[Theorem 4]{JM_15} the convergence under the condition
\[
\sum_{j\in\Z}\abs{j}\ln(1+\abs{j})\abs{\hat{f}_j}^2<\infty.
\]
Using their estimates for $\beta=4$, Lambert \cite[Remark~1.3]{L_21} constructed a $1/2$-Sobolev regular function for which Corollary~\ref{res:cbe_th} fails.

We now turn to the rate of convergence. For a function $f$ on the unit circle, define a seminorm
\[
\normHT{p}{f}^2 = \sum_{k\in\Z}\abs{k}^{2p}\abs{\hat{f}_k}^2,
\]
and introduce the Sobolev space
\[
H^p(\T) = \{f: \normHT{p}{f}<+\infty\},
\] 
which becomes a Hilbert space with the norm
\[
\norm{\cdot}_{H^p(\T)}^2 = \normHT{p}{\cdot}^2+\norm{\cdot}_{L^2(\T)}^2.
\]
\begin{theorem}\label{res:cbe_est}
Let $f\in H^1(\mathbb T)$ and $\beta\le 2$. Then the inequality
\begin{equation}\label{res_eq:cbe_est}
\begin{aligned}
\abs*{\E_\beta^{2n} \left\{\prod_{j=1}^{2n} e^{f(\theta_j)}\right\}\exp\left(-2n\hat{f}_0-\frac{2}{\beta}\sum_{k\ge 1}k\hat{f}_k\hat{f}_{-k}\right)-1}&\\
 \le \frac{4}{\beta^2 n}\exp\left(\frac{1}{\beta}\normHT{1/2}{f}^2 - \frac{2}{\beta}\Re\sum_{k\ge 1}k\hat{f}_k\hat{f}_{-k}\right)&\normHT{1}{f}^2
 \end{aligned}
\end{equation}
holds.
\end{theorem}

Theorem \ref{res:cbe_est} is proved in Subsection \ref{ssec:cbe_est_proof}.

In the special case $\beta=2$, Theorem~\ref{introduction:gessel} relates multiplicative functionals to Schur measures, which were introduced by Okounkov \cite[Section~2.1]{O_01}. The determinantal structure of these measures (see \cite[Theorem~1]{O_01}, \cite[Section~3]{J_01} and the review \cite[Section 5]{BG_12}) yields an exact representation of the left-hand side of \eqref{res_eq:cbe_est} via a Fredholm determinant \cite[Theorem 1]{BO_00}. Theorem~\ref{res:cbe_est} for $\beta=2$ then follows from an operator-theoretic argument.

Finally, we mention several related developments. The loop equation method of Johansson has been further advanced in \cite{BL_18, BG_13, S_13}, where analogous problems are studied for $\beta$-ensembles on the real line with general potentials. Borodin, Gorin and Guionnet proved convergence of additive functionals for discrete $\beta$-ensembles in \cite{BGG_17}. The results of Jiang and Matsumoto were further pursued by Webb \cite{W_16} to establish the rate of the convergence of the power sums $p_j$ in the circular $\beta$-ensembles. For a comprehensive overview of $\beta$-ensembles, see Forrester \cite{F_10}.

\subsection{The Soshnikov Limit Theorem for the sine-$\beta$ process}

For $\beta>0$, the sine-$\beta$ process $\P_\beta$ is a Borel probability measure on the space of configurations --- the space of locally finite subsets of $\R$ (with multiplicities), endowed with the vague topology (see Section \ref{sec:sine_def}). For $\beta=2$, this measure coincides with the sine-process --- the determinantal point process induced by the sine-kernel
\[
K_{\mathcal{S}}(x, y) = \frac{\sin(x-y)}{x-y}.
\]
For $\beta=1, 4$ it may also be defined by its correlation functions expressed via pfaffians.

The sine-$\beta$ process admits several constructions.
Killip and Stoiciu \cite{KS_09} constructed it as a preimage of a randomly shifted lattice under a random monotone function. We review their approach in Appendix.
An alternative construction was developed by Valk\'o and Vir\'ag via the hyperbolic carousel formalism \cite{VV_09}, and,
subsequently, through the spectrum of a random Dirac operator \cite{VV_17}. These constructions are equivalent, as both arise as microscopic limits of the circular $\beta$-ensemble.

The following proposition may be taken as the definition of the sine-$\beta$ process.

\begin{proposition}\label{outline:convergence}
For a continuous compactly supported function $f$ and $\beta\le 2$, we have
\[
\E_\beta^n \left\{\prod_{j=1}^n e^{f(n\theta_j)}\right\} \to \E_\beta\left\{\prod_{x\in X} e^{f(x)}\right\}, \quad \text{ as }n\to\infty.
\]
\end{proposition}
\begin{remark}
Killip and Stoiciu \cite{KS_09} proved this statement for non-positive, smooth, compactly supported functions. Since the set of circular $\beta$-ensembles under the microscopic scaling $\theta\mapsto n\theta$ is tight, their result implies weak convergence; the smoothness assumption can, therefore, be relaxed to continuity.
To avoid the non-positivity restriction, we use the uniform subgaussian estimate from Corollary \ref{res:cbe_th} (see Section~\ref{sec:sine_def}).
\end{remark}

Proposition \ref{outline:convergence} is proved in Section \ref{sec:sine_def}.

An analogue of this convergence for the construction of Valk\'o and Vir\'ag was established in~\cite[Theorem~1, Corollary~2]{VV_20}. We also note that the sine-$\beta$ process satisfies the Gibbs property \cite{DHLM_21}, although it is an open question if this property characterizes the process uniquely. The Gibbs property for a general class of determinantal processes, including the sine process, was established by Bufetov \cite{B_18}.

For a function $f$ on the real line, define
\[
\normHR{p}{f}^2 = \int_\R\abs{\lambda}^{2p}\abs{\hat{f}(\lambda)}^2d\lambda,
\]
and introduce the Sobolev space
\[
H^p(\R) = \{f\in L^2(\R): \normHR{p}{f}<+\infty\},
\] 
which is a Hilbert space if endowed with the norm
\[
\norm{\cdot}_{H^p(\R)}^2 = \normHT{p}{\cdot}^2+\norm{\cdot}_{L^2(\R)}^2.
\]

For a bounded Borel compactly supported function $f$, define the additive functional
\[
S_f(X) = \sum_{x\in X}f(x),\quad X\in\Conf(\R).
\]
By translation invariance, we have $S_f\in L_1(\Conf(\R), \P_\beta)$. The respective regularized additive functional is
\[
\overline{S}_f(X) = S_f(X)-\E_\beta S_f.
\]
We show in Subsection \ref{ssec:regularization} that the variance of $\overline{S}_f$ depends continuously on the $1/2$-Sobolev seminorm of~$f$. This allows one to extend the definition of additive functionals by continuity to all $1/2$-Sobolev regular functions on the real line (see Definition \ref{def:reg}).

Substituting $f(2n\cdot)$ into the inequalities \eqref{res_eq:cbe_est}, \eqref{intro_eq:subgauss} and using Proposition \ref{outline:convergence} yields the following result.

\begin{theorem}\label{res:sb_th}
Let $\beta\le 2$.

1. For $f\in H^1(\R)$, we have
\begin{equation}\label{res_eq:sb_th}
\begin{aligned}
&\abs*{\E_\beta e^{\overline{S}_f}\exp\left( - \frac{2}{\beta}\int_0^\infty\lambda\hat{f}(\lambda)\hat{f}(-\lambda)d\lambda\right)-1} \\
&\le\frac{8}{\beta^2 n}\exp\left(\frac{1}{\beta}\normHR{1/2}{f}^2-\frac{2}{\beta}\Re\int_0^\infty \lambda\hat{f}(\lambda)\hat{f}(-\lambda)d\lambda\right)\normHR{1}{f}^2.
\end{aligned}
\end{equation}

2. For real-valued $f\in H^{1/2}(\R)$, the subgaussian estimate
\[
\E_\beta e^{\overline{S}_f}\le \exp\left(\frac{2}{\beta}\int_0^\infty \lambda\hat{f}(\lambda)\hat{f}(-\lambda)d\lambda\right)
\]
holds.
\end{theorem}
\begin{remark}
The limit variance of $\overline{S}_f$ is invariant under dilations of $f$, while
$$\normHR{p}{f(\cdot/R)} = R^{1/2-p}\normHR{p}{f}.$$
Consequently, the additive functional $\overline{S}_{f(\cdot/R)}$ converge to a Gaussian distribution as $R\to\infty$.
\end{remark}

Theorem \ref{res:sb_th} is proved in Section \ref{sec:cbe_th_proof}.


Theorem \ref{res:sb_th} implies the convergence of additive functionals to the Gaussian distribution with respect to the Kolmogorov-Smirnov metric. For a real-valued function $f$, denote
\[
F_f(x) = \P_\beta(\overline{S}_f\le x), \quad F_{\mathcal{N}}(x)=\frac{1}{\sqrt{2\pi}}\int_{-\infty}^xe^{-\frac{t^2}{2}}dt.
\]

\begin{corollary}\label{res:KS_conv}
For real-valued $f\in H^1(\R)$ and $\beta \le 2$, there exists a constant $C$ such that for any $R>0$, the inequality
\[
\sup_x\abs{F_{f(\cdot/R)}-F_{\mathcal{N}}} \le \frac{C}{\sqrt{\ln R}}
\]
holds, provided the function $f$ is normalized by
\[
\int_0^\infty \lambda\hat{f}(\lambda)\hat{f}(-\lambda)d\lambda = \frac{\beta}{4}.
\]
\end{corollary}

Corollary \ref{res:KS_conv} is proved in Section \ref{sec:KS_conv}.

Corollary \ref{res:cbe_th} is established under the optimal condition of $1/2$-Sobolev regularity. A natural question is whether this regularity condition remains sufficient for the sine-$\beta$ process.
\begin{theorem}\label{res:opt_conv}
For $\beta\le 2$ and $f\in H^{1/2}(\R)$, the convergence
\[
\E_\beta e^{\overline{S}_{f(\cdot/R)}} \to \exp\left(\frac{2}{\beta}\int_0^\infty\lambda\hat{f}(\lambda)\hat{f}(-\lambda)d\lambda\right) \quad \text{ as }R\to\infty
\]
holds.
\end{theorem}
\begin{remark}
The condition $f\in L^2(\R)$ can be omitted, and one may consider functions that have a finite $1/2$-Sobolev seminorm (with the Fourier transform taken in the distributional sense). In this case, though, one has to correctly define the regularization of additive functionals (see remark to Definition \ref{def:reg}).
\end{remark}
\begin{remark}
To prove Theorem \ref{res:opt_conv}, one only needs the second claim of Theorem \ref{res:sb_th} and convergence for a dense subset of the space $H^{1/2}(\R)$. The latter was established by Lebl\'e \cite[Theorem 1]{L_20} and Lambert \cite[Theorem 1.10]{L_21}.
\end{remark}

Theorem \ref{res:opt_conv} is proved in Subsection \ref{ssec:opt_proof}.

Our proof relies only on the identification of the sine-$\beta$ process as the scaling limit of the circular $\beta$-ensemble and is, in a sense, independent of previous constructions of $\P_\beta$. More precisely, even without the results of Killip and Stoiciu \cite{KS_09} or Valk\'o and Vir\'ag \cite{VV_20}, Theorems \ref{res:sb_th}, \ref{res:KS_conv}, \ref{res:opt_conv} and Proposition \ref{outline:convergence} would still hold for any weak subsequential limit of the circular $\beta$-ensemble under the microscopic scaling. Such limits exist by tightness (see Section \ref{sec:sine_def}). The role of the aforementioned works is therefore to identify these limits uniquely.

For $\beta=2$, the convergence of additive functionals $\overline{S}_{f(\cdot/R)}$ to the Gaussian distribution was established by Soshnikov \cite[Theorem 1]{S_00}. Theorem \ref{res:KS_conv} for $\beta=2$ was proved by Bufetov \cite{B_25}.

Let us list some of related works.
The convergence of additive functionals under the sine-$\beta$ process for all~$\beta>0$ has previously been established by Lebl\'e \cite[Theorem 1]{L_20} from the Gibbs property \cite[Theorem~1.1]{DHLM_21}. Another proof is due to Lambert \cite[Theorem 1.10]{L_21}. Lambert relied on the coupling of the sine-$\beta$ process with the circular $\beta$-ensemble, introduced by Valk\'o and Vir\'ag \cite[Theorem 1]{VV_20}. Both methods require the function to be compactly supported. This restriction is, however, unnatural, as may be seen with the example of~$\beta=2$, for which the exact formula is available --- the scaling limit of the Borodin-Okounkov formula~\cite{BC_03, B_24}. This formula implies Theorems~\ref{res:sb_th}, \ref{res:KS_conv}, \ref{res:opt_conv} for $\beta=2$ and, in particular, yields that the convergence of~$\overline{S}_{f(\cdot/R)}$ depends not on the decay of the function $f$ itself, but rather on the decay of its Fourier transform. Our method avoids the compact support restriction by deriving an explicit inequality~\eqref{res_eq:sb_th}, which may be extended by continuity to all functions, for which the $1$-Sobolev seminorm is finite.

We note that the dependence of the variance of additive functionals on the seminorm, rather than the norm, is connected, following Ghosh and Peres \cite[Theorem 6.1]{GP_17}, to the rigidity of the process \cite{G_15, GP_17}. In particular, this property allows one to approximate an indicator function by a function with arbitrarily small variance of the respective additive functional. This was used by Chhaibi and Najnudel \cite[Theorem~1.2]{CN_18} to show the rigidity of the sine-$\beta$ process.

\subsection{Structure of the paper}
The rest of the paper is structured as follows. We recall the definition of the Jack polynomials in Subsection~\ref{sec:jack}. Next we prove Theorem~\ref{intro:gessel} and Corollary~\ref{res:cbe_th} in Subsection~\ref{ssec:cbe_th_proof}. We conclude the proof of Theorem~\ref{res:cbe_est} in Subsection~\ref{ssec:cbe_est_proof}. In Section~\ref{sec:cbe_th_proof}, we use the former results and Proposition~\ref{outline:convergence} to prove Theorem~\ref{res:sb_th}. We show how Theorem~\ref{res:opt_conv} follows from Theorem~\ref{res:sb_th} in Subsection~\ref{ssec:opt_proof}. Section~\ref{sec:sine_def} is devoted to the proof of Proposition~\ref{outline:convergence}. Finally, we show how Corollary~\ref{res:KS_conv} follows from Theorem~\ref{res:sb_th} in Section~\ref{sec:KS_conv}.

\subsection{Acknowledgements}
The author is a winner of the BASIS Foundation Competition and is deeply grateful to the Jury
and the sponsors. The author is a winner of the all-Russia mathematical August Moebius contest of graduate and undergraduate student papers and thanks the jury and the board for the high praise of his work.

\section{Proof of Theorem \ref{intro:gessel}, Corollary \ref{res:cbe_th} and Theorem \ref{res:cbe_est}}

\subsection{Jack polynomials}\label{sec:jack}
This section recalls the definition and properties of the Jack polynomials, following Macdonald \cite{M_79}. A brief introduction to the theory of symmetric functions is contained in \cite[Chapter 2]{BO_17}, \cite[Section 2]{BG_12}.

\textbf{Partitions}

A \textit{partition} $\lambda = (\lambda_1, \lambda_2, \ldots)$ is a non-increasing sequence of nonnegative integers, $\lambda_1 \ge \lambda_2 \ge \cdots$, with only finitely many non-zero elements. The number of non-zero parts is the \textit{length} $l(\lambda)$, and the sum of the parts is the \textit{size} $\abs{\lambda} = \sum_j \lambda_j$. The \textit{dominance order} is a partial order on partitions of fixed size: we say that $\lambda \ge \mu$ if
\[
\sum_{j=1}^k \lambda_j\ge \sum_{j=1}^k \mu_j, \quad \text{ for all }k\ge 0.
\]

\textbf{Symmetric functions}

The algebra of symmetric functions in finitely many variables is $$\Lambda_N = \mathbb{C}[x_1, \dots, x_N]^{S_N},$$ graded by degree.
A projective limit with respect to the natural projections $\pi_N^{N+k}: \Lambda_{N+k} \to \Lambda_N$ (setting~$x_{N+1}, \dots, x_{N+k}$ to zero) yields the \textit{algebra of symmetric functions} in infinitely many variables,
\[
\Lambda = \bigoplus_{n=0}^\infty \Lambda^n, \quad \Lambda^n= \varprojlim_N \Lambda_N^n.
\]
An element of $\Lambda$ may be viewed as a symmetric formal power series in infinitely many variables whose terms have bounded degree.

A fundamental multiplicative basis is provided by the \textit{Newton power sums}
\[
p_k = \sum_{j=1}^\infty x_j^k, \in \Lambda\quad k\in \N,
\]
with $p_0=1$. For the purposes of this paper, it is most convenient to regard $\Lambda$ as a polynomial ring in the variables $\{p_j\}_{j\ge 1}$.

\begin{theorem}\label{JP:newton_basis}
The Newton power sums form a multiplicative basis of $\Lambda$:
\[
\Lambda\simeq \C[p_0, p_1, \ldots].
\]
\end{theorem}

For a partition $\lambda$, set $p_\lambda=\prod_jp_{\lambda_j}$. 
For a real parameter $\alpha > 0$, we define a scalar product $\alsc{\cdot, \cdot}$ on $\Lambda$ by declaring the power-sum basis to be orthogonal:
\begin{equation}\label{JP_eq:scalar}
\alsc{p_\lambda, p_\mu} = \delta_{\lambda\mu}z_\lambda \alpha^{l(\lambda)}, \quad \text{ where }z_\lambda = \prod_{j\ge 1}j^{m_j}m_j!,
\end{equation}
and $m_j$ is the multiplicity of $j$ in $\lambda$. This scalar product will be used to define the Jack polynomials.

\textbf{Schur polynomials}

The \textit{Schur polynomials} form another important basis of $\Lambda$. For $N \ge l(\lambda)$, they can be defined as a ratio of determinants
\[
s_\lambda (x_1, \ldots, x_N) = \frac{\det (x_i^{\lambda_j+N-j})_{i, j=1}^k}{\det (x_i^{N-j})_{i, j=1}^N} \in \Lambda_N.
\]
These are compatible with the natural projections and thus define elements of $\Lambda$. Moreover, these polynomials are homogeneous, satisfy $\deg s_\lambda = \abs{\lambda}$, and are orthonormal with respect to $\langle \cdot, \cdot \rangle_1$.

\textbf{Jack polynomials}

The \textit{Jack polynomials} $\J{\lambda}$ arise as a one-parameter deformation of the Schur polynomials. More concretely, one applies the Gram-Schmidt orthogonalization procedure to the Schur polynomials with respect to the inner product $\alsc{\cdot, \cdot}$. When the partitions are ordered by the dominance order, the resulting orthogonal basis consists of the Jack polynomials. Since that order is only partial, it is non-trivial that Jack polynomials are orthogonal.

\begin{theorem}\label{Jack:def}
There exists a unique family $\{\J{\lambda}\}_\lambda$ of symmetric functions such that
\begin{itemize}
\item $\alsc{\J{\lambda}, \J{\mu}} = 0$ for $\lambda \ne \mu$,
\item $\J{\lambda} = s_\lambda + \sum_{\mu<\lambda}C_{\lambda\mu}s_\mu$ for some $C_{\lambda\mu}\in\mathbb{Q}(\alpha)$.
\end{itemize}
We call $\J{\lambda}$ Jack polynomials.
\end{theorem}
\begin{remark}
Usually the second property is given in terms of monomial symmetric functions. Observe, however, that the definition above is equivalent since $J^{(1)}_{\lambda}=s_\lambda$.
\end{remark}

\textbf{Properties of Jack polynomials}

A symmetric function in $n$ variables can be evaluated on an $n$-tuple $(z_1, \dots, z_n)$ of points on the unit circle $\mathbb{T}$. We equip $\Lambda_n$ with the inner product
\begin{equation}\label{JP_eq:scalar_n}
\alscn{u, v}=\cbee{2/\alpha}u\bar{v}.
\end{equation}
Let $\pi_n: \Lambda \to \Lambda_n$ be the canonical projection that sets $x_{n+1}, x_{n+2}, \dots$ to zero. For a partition $\lambda$, define the factor
\begin{equation}\label{JP_eq:A_def}
\A = \prod_{(i, j)\in \lambda}\left(1 - \frac{\alpha - 1}{n+j\alpha - i}\right),
\end{equation}
where $(i,j)$ runs over the cells of the Young diagram of $\lambda$ (i.e., $i\ge 1$, $1 \le j \le \lambda_i$).

\begin{proposition}[{\cite[(10.22), (10.35), (10.37)]{M_79}, \cite[Lemma 4.2]{JM_15}}]\label{Jack:orth}
Jack polynomials are orthogonal with respect to the circular $\beta$-ensemble:
\[
\frac{\alscn{\pi_n (\J{\lambda}), \pi_n(\J{\mu})}}{\alsc{\J{\lambda}, \J{\lambda}}} = \delta_{\lambda, \mu}\delta_{l(\lambda)\le n}\A.
\]
\end{proposition}

Finally, we recall the Cauchy identity.

\begin{proposition}[{\cite[p. 309, p. 377]{M_79}}]\label{JP:cauchy}
We have the equality
\[
\sum_{\abs{\lambda}=k}\frac{\J{\lambda}(x)\J{\lambda}(y)}{\alsc{\J{\lambda}, \J{\lambda}}} = \sum_{\abs{\lambda}=k}\frac{p_\lambda(x)p_\lambda(y)}{z_\lambda\alpha^{l(\lambda)}}.
\]
Summing over all $k\ge 0$ yields the Cauchy identity
\[
\sum_{\lambda}\frac{\J{\lambda}(x)\J{\lambda}(y)}{\alsc{\J{\lambda}, \J{\lambda}}} = \exp\left(\frac{1}{\alpha}\sum_{j=1}^\infty \frac{p_j(x), p_j(y)}{j}\right).
\]
\end{proposition}

In the following section, we show how these two statements imply the Gessel-type expansion (Theorem~\ref{intro:gessel}).

\subsection{Proof of Theorem \ref{intro:gessel} and Corollary \ref{res:cbe_th}}\label{ssec:cbe_th_proof}
\begin{proof}[Proof of Theorem \ref{intro:gessel}]

We begin by expanding the function $\exp\left(\sum_{j=1}^n f(\theta_j)\right)$, which is symmetric in the variables $z_j = e^{i\theta_j}$, into a series of Jack polynomials using the Cauchy identity. The desired Gessel-type formula will then emerge from this expansion after we apply Proposition~\ref{Jack:orth} to compute the expectations. 

\textbf{Jack polynomial expansion of $\exp\left(\sum_{j=1}^n f(\theta_j)\right)$}

First, we decompose $f$ into a constant term and its analytic parts inside and outside the unit circle. For~$z = e^{i\theta}$, we write
\begin{equation}\label{CLT_eq:wh_fact}
f(\theta) = \hat{f}_0 + \sum_{k\ge 1}\hat{f}_k z^k + \sum_{k\ge 1} \hat{f}_{-k}z^{-k} = \hat{f}_0+f_+(\theta)+f_-(\theta), \quad z=e^{i\theta}.
\end{equation}
Recall that by $\pi_n:\Lambda\to\Lambda_n$ we denoted the canonical projection. Using this projection, the exponential of the sum of functionals can be rewritten in terms of power sums $p_j$. A direct calculation yields
\[
\prod_{j=1}^n e^{f(\theta_j)} = e^{n\hat{f}_0}\exp\left(\alpha^{-1}\sum_{j =1}^\infty \frac{p_j(\rho_+)(\pi_np_j)(z_1, \ldots, z_n)}{j}+\alpha^{-1}\sum_{j=1}^\infty\frac{p_j(\rho_-)(\pi_np_j)(z_1^{-1}, \ldots, z_n^{-1})}{j}\right), \quad z_j=e^{i\theta_j},
\]
where $p_j(\rho_\pm) = \alpha j\hat{f}_{\pm j}$ for $j\ge 1$. These are precisely the homomorphisms defined in Theorem \ref{intro:gessel}.

Since the specializations $\rho_\pm$ and the projections $\pi_n$ are algebra homomorphisms, we can apply the Jack-polynomial Cauchy identity (Proposition \ref{JP:cauchy}) to expand each of the two exponential factors. This gives
\[
\exp\left(\alpha^{-1}\sum_{j =1}^\infty \frac{p_j(\rho_\pm)(\pi_np_j)(z_1^{\pm 1}, \ldots, z_n^{\pm 1})}{j}\right) = \sum_{\lambda}\frac{\J{\lambda}(\rho_\pm)(\pi_n\J{\lambda})(z_1^{\pm 1}, \ldots, z_n^{\pm 1})}{\alsc{\J{\lambda}, \J{\lambda}}},
\]
where, we recall, $\J{\lambda}$ denotes the Jack polynomial associated with the partition $\lambda$ (see Theorem \ref{Jack:def}), and~$\alsc{\cdot, \cdot}$ is the inner product \eqref{JP_eq:scalar}.

By the definition Jack polynomials are homogeneous of degree $\abs{\lambda}$. Consequently, the product of the two series above, when multiplied by $e^{n\hat{f}0}$, gives a power series expansion of $\prod_{j=1}^n e^{f(\theta_j)}$. If we assume that $f$ is holomorphic in a neighborhood of the unit circle, then this series converges uniformly on $\mathbb{T}^n$. Let us first proceed under this assumption.

\textbf{Proof for a holomorphic function $f$}

Uniform convergence permits us to interchange the sum over partitions with the integral defining the expectation $\mathbb{E}_{2/\alpha}^n$. The desired result then follows from the chain of equalities below:
\begin{equation}\label{CLT_eq:gessel_exp}
\begin{aligned}
e^{-n\hat{f}_0}\E_{2/\alpha}^n \prod_{j=1}^n e^{f(\theta_j)} = \sum_{\lambda, \mu}\frac{\J{\lambda}(\rho_+)\J{\mu}(\rho_-)}{\alsc{\J{\lambda}, \J{\lambda}}}\frac{\alscn{\pi_n\J{\lambda}, \pi_n\J{\mu}}}{\alsc{\J{\mu}, \J{\mu}}}&\\
 = \sum_{l(\lambda)\le n}\frac{\J{\lambda}(\rho_+)\J{\mu}(\rho_-)}{\alsc{\J{\lambda}, \J{\lambda}}}\A&,
 \end{aligned}
\end{equation}
where $\alscn{\cdot, \cdot}$ is defined in the formula \eqref{JP_eq:scalar_n} and the quantity $\A$ in the formula \eqref{JP_eq:A_def}.

The first equality is obtained by interchanging summation and integration and by using the property
\[
(\pi_n\J{\lambda})(z_1^{-1}, \ldots, z_n^{-1})= \overline{(\pi_n\J{\lambda})(z_1, \ldots, z_n)}, \quad \text{ for }z_j\in\T,
\]
which holds because $\J{\lambda}$ belongs to $\mathbb{Q}(\alpha)[p_0, p_1, \dots]$ by Theorem \ref{Jack:def}.

The second equality follows from substituting the result of Proposition \ref{Jack:orth}.

This completes the proof of Theorem \ref{intro:gessel} for functions whose Fourier coefficients decay exponentially.

\textbf{Absolute convergence of the Jack polynomial expansion}

Before extending our result to all functions in the $H^{1/2}(\T)$ Sobolev space, we first establish the second claim of the theorem concerning the absolute convergence of the Jack-polynomial series. Observe that $\A\le 1$ for $\alpha\ge 1$. Therefore, by the Cauchy-Bunyakovsky-Schwarz inequality, the Jack polynomial expansion is bounded by
\[
\sum_{\lambda} \abs*{\frac{\J{\lambda}(\rho_+)\J{\lambda}(\rho_-)}{\alsc{\J{\lambda}, \J{\lambda}}}}\le \sqrt{\sum_{\lambda} \frac{\J{\lambda}(\rho_+)\overline{\J{\lambda}(\rho_+)}}{\alsc{\J{\lambda}, \J{\lambda}}}}\sqrt{\sum_{\lambda} \frac{\J{\lambda}(\rho_-)\overline{\J{\lambda}(\rho_-)}}{\alsc{\J{\lambda}, \J{\lambda}}}}.
\]
From Proposition \ref{JP:cauchy}, we know that the right-hand side is finite if and only if
\[
\sum_{j\ge 1}\frac{\abs{p_j(\rho_\pm)}^2}{j} = \alpha^2\sum_{j\ge 1}j\abs{\hat{f}_{\pm j}}^2<+\infty.
\]
This condition is precisely the definition of $f$ belonging to the space $H^{1/2}(\mathbb{T})$. Therefore, the series converges absolutely under the theorem's hypothesis.

Under the assumption that $f$ is holomorphic and real-valued, we have
\begin{equation}\label{CLT_eq:subg}
\begin{aligned}
e^{-n\hat{f}_0}\E_{2/\alpha}^n \left[ \prod_{j=1}^n e^{f(\theta_j)} \right] = \sum_{l(\lambda)\le n} \frac{\J{\lambda}(\rho_+)\J{\lambda}(\rho_-)}{\alsc{\J{\lambda}, \J{\lambda}}}\A&\\
 \le \sum_\lambda\frac{\J{\lambda}(\rho_+)\J{\lambda}(\rho_-)}{\alsc{\J{\lambda}, \J{\lambda}}} = \exp\left(\frac{\alpha}{2}\normHT{1/2}{f}^2\right)&.
\end{aligned}
\end{equation}
Indeed, all of the terms in the Jack polynomial expansion are positive because $\hat{f}_{-j}=\overline{\hat{f}_j}$, so the middle inequality holds. The rightmost equality follows from Proposition \ref{JP:cauchy}.

\textbf{Proof for a general function $f\in H^{1/2}(\T)$}

We use a regularization argument. For $q \in (0,1)$, define a smoothed version of $f$:
\[
f_q(\theta) = \hat{f}_0 + \sum_{k\ge 1}q^k\hat{f}_kz^k + \sum_{k\ge 1}q^k\hat{f}_{-k}z^{-k}, \quad z=e^{i\theta}, q\in(0, 1).
\]
The function $f_q$ is holomorphic in a neighborhood of the unit circle, so its corresponding specializations $\rho_\pm^q$ (where $p_j(\rho_\pm^q) = q^j p_j(\rho_\pm)$) satisfy the expansion \eqref{CLT_eq:gessel_exp}. Our goal is to take the limit $q \to 1-$.

First, consider the right-hand side of \eqref{CLT_eq:gessel_exp} for $f_q$. Since Jack polynomials are homogeneous of degree $|\lambda|$, we have $J_\lambda(\rho_\pm^q) = q^{|\lambda|} J_\lambda(\rho_\pm)$. Using that $0 \le \mathcal{A}(\lambda) \le 1$ for $\alpha \ge 1$, we obtain the bound
\[
\sum_{l(\lambda)\le n}\abs*{\frac{\J{\lambda}(\rho_+^q)\J{\mu}(\rho_-^q)}{\alsc{\J{\lambda}, \J{\lambda}}}\A} = \sum_{l(\lambda)\le n}q^{2\abs{\lambda}}\abs*{\frac{\J{\lambda}(\rho_+)\J{\mu}(\rho_-)}{\alsc{\J{\lambda}, \J{\lambda}}}\A}\le \sum_{l(\lambda)\le n}\abs*{\frac{\J{\lambda}(\rho_+)\J{\mu}(\rho_-)}{\alsc{\J{\lambda}, \J{\lambda}}}}.
\]
The rightmost series converges by the second claim of the theorem. The Dominated Convergence Theorem then allows us to interchange the limit $q \to 1$ with the summation, showing that the right-hand side of \eqref{CLT_eq:gessel_exp} for $f_q$ converges to the same expression for the original $f$.

It remains to show that the left-hand side of \eqref{CLT_eq:gessel_exp} converges
$$
\E_{2/\alpha}^n \left[ \prod_{j=1}^n e^{f_q(\theta_j)} \right]\to \E_{2/\alpha}^n \left[ \prod_{j=1}^n e^{f(\theta_j)} \right], \quad \text{ as }q\to 1-.
$$

As $q\to 1-$, the sequence $f_q$ converges to $f$ in both $L^2(\T)$ and $H^{1/2}(\T)$. We may extract an almost surely convergent subsequence $f_{q_k}$, to which we apply Fatou's lemma and, using the inequality \eqref{CLT_eq:subg}, obtain
\begin{equation}\label{CLT_eq:subgaussianity}
\E_{2/\alpha}^n\abs*{\prod_{j=1}^ne^{f(\theta_j)}} = \E_{2/\alpha}^n\prod_{j=1}^ne^{\Re f(\theta_j)}\le\liminf_{k\to\infty} \E_{2/\alpha}^n\prod_{j=1}^ne^{\Re f_{q_k}(\theta_j)} \le e^{\frac{\alpha}{2}\normHT{1/2}{\Re f}^2+n\hat{f}_0}.
\end{equation}
By a standard argument, that subgaussian estimate implies the inequality
\begin{equation}\label{CLT_eq:var_est}
\E_{2/\alpha}^n\abs*{\sum_{j=1}^nf(\theta_j)-n\hat{f}_0}^2 \le C\normHT{1/2}{f}^2
\end{equation}
for some independent of $f$ constant.

To conclude the proof, we apply the Cauchy-Bunyakovsky-Schwarz inequality
\[
\E_{2/\alpha}\abs*{\prod_{j=1}^n e^{f_q(\theta_j)} - \prod_{j=1}^n e^{f(\theta_j)}} \le \sqrt{\E_{2/\alpha}^n \prod_{j=1}^n e^{2\Re(f_q(\theta_j)+f(\theta_j))}}\sqrt{\E_{2/\alpha}^n\abs*{\sum_{j=1}^n f_q(\theta_j)-f(\theta_j)}^2}.
\]
By the estimate \eqref{CLT_eq:subgaussianity}, the first factor on the right-hand side stays bounded as $q\to 1-$, while the second one converges to $0$ due to \eqref{CLT_eq:var_est}. Thus the right-hand side of \eqref{CLT_eq:gessel_exp} for $f_q$ converges to the one for $f$. Theorem \ref{intro:gessel} is proved.
\end{proof}

\begin{proof}[Proof of Corollary \ref{res:cbe_th}]
The second claim of Corollary \ref{res:cbe_th} is exactly the inequality \eqref{CLT_eq:subgaussianity}. The first claim then follows from the convergence of the Jack polynomial expansion
\[
e^{-n\hat{f}_0}\E_{2/\alpha}^n\prod_{j=1}^ne^{f(\theta_j)} = \sum_{l(\lambda)\le n}\frac{\J{\lambda}(\rho_+)\J{\mu}(\rho_-)}{\alsc{\J{\lambda}, \J{\lambda}}}\A \to \sum_{\lambda}\frac{\J{\lambda}(\rho_+)\J{\mu}(\rho_-)}{\alsc{\J{\lambda}, \J{\lambda}}},
\]
as $n\to \infty$, since $\A\to 1$. By Proposition \ref{JP:cauchy}, the limit equals
\[
\sum_{\lambda}\frac{\J{\lambda}(\rho_+)\J{\mu}(\rho_-)}{\alsc{\J{\lambda}, \J{\lambda}}} = \exp\left(\alpha^{-1}\sum_{j\ge 1}\frac{p_j(\rho_+)p_j(\rho_-)}{j}\right) = \exp\left(\alpha\sum_{j\ge 1}j\hat{f}_j\hat{f}_{-j}\right).
\]
This completes the proof of Corollary \ref{res:cbe_th}.
\end{proof}

\subsection{Proof of Theorem \ref{res:cbe_est}}\label{ssec:cbe_est_proof}

To establish Theorem \ref{res:cbe_est}, we determine the average size of a random partition under the Jack measure.
\begin{lemma}[{\cite[Proposition 1.3.1]{M_23}}]\label{CLT:diag_exp}
For a real-valued function $f\in H^{1}(\T)$, we have
\begin{equation}\label{CLT_eq:diag_exp}
\sum_{\lambda}\abs{\lambda}\frac{\J{\lambda}(\rho_+)\J{\lambda}(\rho_-)}{\alsc{\J{\lambda},\J{\lambda}}} = \frac{\alpha}{2}\normHT{1}{f}^2\exp\left(\frac{\alpha}{2}\normHT{1/2}{f}^2\right)
\end{equation}
\end{lemma}
\begin{proof}
Replacing $p_k$ with $q^kp_k$ in Proposition \ref{JP:cauchy} yields the formal series identity
\[
\sum_{\lambda}\frac{q^{\abs{\lambda}}p_\lambda(x)p_\lambda(y)}{z_\lambda\alpha^{l(\lambda)}} = \exp\left(\frac{1}{\alpha}\sum_{k\in\N}\frac{q^kp_k(x)p_k(y)}{k}\right).
\]
Next one differentiates both sides with respect to $q$ and evaluates at $q=1$ to deduce
\[
\sum_{\lambda}\abs{\lambda}\frac{p_\lambda(x)p_\lambda(y)}{z_\lambda\alpha^{l(\lambda)}} = \frac{1}{\alpha}\exp\left(\frac{1}{\alpha}\sum_{k\in\N}\frac{p_k(x)p_k(y)}{k}\right)\left(\sum_{k\in\N}p_k(x)p_k(y)\right).
\]
By the first claim of Proposition \ref{JP:cauchy}, the left-hand coincides with the left-hand side of \eqref{CLT_eq:diag_exp}. Substituting the specializations $\rho_\pm$ into the above identity concludes the desired equality.
\end{proof}

\begin{proof}[Proof of Theorem \ref{res:cbe_est}]
Throughout this proof set $\alpha=2/\beta$.

Using Theorem \ref{intro:gessel} and Proposition \ref{JP:cauchy}, we bound the difference between the expectation and its large-$n$ limit by three sums over partitions of different lengths
\begin{align*}
\abs*{\E_\beta^{2n} \left\{\prod_{j=1}^{2n} e^{f(\theta_j)-\hat{f}_0}\right\}-\exp\left(\frac{2}{\beta}\sum_{k\ge 1}k\hat{f}_k\hat{f}_{-k}\right)} \le \sum_{l(\lambda) \le n}\abs*{\frac{\J{\lambda}(\rho_+)\J{\lambda}(\rho_-)}{\alsc{ \J{\lambda},\J{\lambda}}}(\mathcal{A}_\lambda^\alpha(2n)-1)}\\
+ \sum_{n+1\le l(\lambda) \le 2n}\abs*{\frac{\J{\lambda}(\rho_+)\J{\lambda}(\rho_-)}{\alsc{ \J{\lambda},\J{\lambda}}}(\mathcal{A}_\lambda^\alpha(2n)-1)} + \sum_{l(\lambda) \ge 2n+1}\abs*{\frac{\J{\lambda}(\rho_+)\J{\lambda}(\rho_-)}{\alsc{ \J{\lambda},\J{\lambda}}}}.
\end{align*}
Since $0\le \mathcal{A}_\lambda^\alpha(2n)\le 1$ for $\alpha\ge 1$, the last two sums are bounded by
\[
Z_{\ge n+1} = \sum_{l(\lambda) \ge n+1}\abs*{\frac{\J{\lambda}(\rho_+)\J{\lambda}(\rho_-)}{\alsc{ \J{\lambda},\J{\lambda}}}}.
\]
We denote the first sum by $Z_{\le n}$. In what follows we establish the inequalities
\begin{align}
Z_{\le n} \le \frac{\alpha(\alpha-1)}{n}\exp\left(\frac{\alpha}{2}\normHT{1/2}{f}^2\right)\normHT{1}{f}^2,\label{CLT_eq:small_p}\\
Z_{\ge n+1} \le \frac{\alpha}{n}\exp\left(\frac{\alpha}{2}\normHT{1/2}{f}^2\right)\normHT{1}{f}^2,\label{CLT_eq:large_p}
\end{align}
which imply the desired claim.

\textbf{Proof of \eqref{CLT_eq:large_p}}

This inequality follows from the Rankin trick and the Cauchy-Bunyakovsky-Schwarz inequality
\begin{align*}
\sum_{l(\lambda) \ge n+1}\abs*{\frac{\J{\lambda}(\rho_+)\J{\lambda}(\rho_-)}{\alsc{ \J{\lambda},\J{\lambda}}}} \le \frac{1}{n}\sum_{\lambda}\abs{\lambda}\abs*{\frac{\J{\lambda}(\rho_+)\J{\lambda}(\rho_-)}{\alsc{ \J{\lambda},\J{\lambda}}}}
\\ \le \frac{1}{n}\sqrt{\sum_{\lambda}\abs{\lambda}\frac{\J{\lambda}(\rho_+)\overline{\J{\lambda}(\rho_+)}}{\alsc{ \J{\lambda},\J{\lambda}}}}\sqrt{\sum_{\lambda}\abs{\lambda}\frac{\J{\lambda}(\rho_-)\overline{\J{\lambda}(\rho_-)}}{\alsc{ \J{\lambda},\J{\lambda}}}}.
\end{align*}
Following the notation in \eqref{CLT_eq:wh_fact}, we decompose our function $f=\hat{f}_0 + f_++f_-$ into its constant term and analytic parts inside and outside the unit circle. In this notation the pairs of specializations $(\rho_\pm, \overline{\rho_\pm})$ correspond (in accordance with the definition given in Theorem \ref{intro:gessel}) to the functions $2\Re f_\pm$. Therefore, application of Lemma \ref{CLT:diag_exp} concludes the proof
\begin{align*}
\sum_{\lambda}\abs{\lambda}\frac{\J{\lambda}(\rho_\pm)\overline{\J{\lambda}(\rho_\pm)}}{\alsc{ \J{\lambda},\J{\lambda}}} = 2\alpha\normHT{1}{\Re f_\pm}^2\exp\left(2\alpha\normHT{1/2}{\Re f_\pm}^2\right)\\
\le \alpha\normHT{1}{f}^2\exp\left(\alpha\normHT{1/2}{f_\pm}^2\right).
\end{align*}

\textbf{Proof of \eqref{CLT_eq:small_p}}

For $a_1, \ldots, a_k\in(0, 1)$, we have the inequality
\[
1-\prod_{j=1}^k(1-a_j)\le \sum_{j=1}^ka_j.
\]
Consequently, if $l(\lambda)\le n$, the bound
\[
\abs{\mathcal{A}_\lambda^\alpha(2n)-1} \le (\alpha-1)\sum_{(m, l)\in\lambda} \frac{1}{2n+l\alpha-m} \le \frac{\alpha-1}{n}\abs{\lambda}
\]
holds.

We substitute this estimate into the sum and use the Cauchy-Bunyakovsky-Schwarz inequality to obtain
\begin{align*}
\sum_{l(\lambda) \le n}\abs*{\frac{\J{\lambda}(\rho_+)\J{\lambda}(\rho_-)}{\alsc{ \J{\lambda},\J{\lambda}}}(\mathcal{A}_\lambda^\alpha(2n)-1)} \le\frac{\alpha -1}{n}\sum_{\lambda\in\Y}\abs{\lambda}\abs*{\frac{\J{\lambda}(\rho_+)\J{\lambda}(\rho_-)}{\alsc{ \J{\lambda},\J{\lambda}}}}\\
\le \frac{\alpha-1}{n}\sqrt{\sum_{\lambda\in\Y}\abs{\lambda}\frac{\J{\lambda}(\rho_+)\overline{\J{\lambda}(\rho_+)}}{\alsc{ \J{\lambda},\J{\lambda}}}}\sqrt{\sum_{\lambda\in\Y}\abs{\lambda}\frac{\J{\lambda}(\rho_-)\overline{\J{\lambda}(\rho_-)}}{\alsc{ \J{\lambda},\J{\lambda}}}}.
\end{align*}
Last, as above, Lemma \ref{CLT:diag_exp} is applied to the rightmost expression. Inequality \eqref{CLT_eq:small_p} is proved.
\end{proof}

\section{Proof of Theorem \ref{res:sb_th}}\label{sec:cbe_th_proof}

In this section we explain how Theorem \ref{res:sb_th} follows from Theorem \ref{res:cbe_est}, Corollary \ref{res:cbe_th} and Proposition \ref{outline:convergence}. We first proceed under more restrictive assumptions on $f$.
\begin{lemma}\label{sb_CLT:preliminary}
Theorem \ref{res:sb_th} holds for compactly supported smooth functions on the real line.
\end{lemma}
\begin{proof}
To avoid confusion, we clarify the notation used throughout the proof.
We consider the functions $f(j\cdot)$, with $j\in\N$, as functions on both the real line and, through the identification $f(j\theta)$, with $z=e^{i\theta}$ and $\theta\in(-\pi, \pi)$, on the unit circle. When $f$ is smooth and compactly supported, the function $f(j\theta)$ is smooth on $\T$ for sufficiently large $j$ (specifically, once the support of $f(j\cdot)$ is contained in $[-\pi, \pi]$).
 
We denote by $\widehat{f(j\cdot)}_k$ the $k$-th Fourier coefficient when the function is viewed as a function on the unit circle $\T$, and by $\hat{f}(\lambda)$ the Fourier transform at $\lambda\in \R$ when $f$ is considered on $\R$. For sufficiently large $j$, these are related by the identity:
\begin{equation}\label{sb_CLT_eq:fourier}
\widehat{f(j\cdot)}_k = \frac{1}{j}\hat{f}\left(\frac{k}{j}\right).
\end{equation}

We turn to the proof of the lemma. We substitute the sequence of functions $f(2n\cdot)$ into the inequality \eqref{res_eq:cbe_est} and consider the limit $n\to\infty$. Proposition \ref{outline:convergence} gives that the expectations converge
\[
\E_\beta^{2n} \left\{\prod_{j=1}^{2n} e^{f(\theta_j)}\right\}\to \E_\beta e^{S_f},\quad \text{ as }n\to\infty.
\]
To establish the convergence of the left-hand side of inequality \eqref{res_eq:cbe_est} to the left-hand side of \eqref{res_eq:sb_th}, we verify the convergence of the expressions entering the former. A direct calculation yields this convergence for the average
\[
2n\widehat{f(2n\cdot)}_0 = \frac{1}{2\pi}\int_{-\pi n}^{\pi n}f(x)dx \to \frac{1}{2\pi}\int_\R f(x)dx = \E_\beta S_f.
\]
The convergence of the Gaussian factor, which is, by the identity \eqref{sb_CLT_eq:fourier}, equal to the exponential of
\[
\sum_{j\ge 0}j\widehat{f(2n\cdot)}_{j}\widehat{f(2n\cdot)}_{-j} = \sum_{j\ge 0}\frac{j}{(2n)^2}\hat{f}\left(\frac{j}{2n}\right)\hat{f}\left(-\frac{j}{2n}\right),
\]
follows from this factor being the Darboux sum for the integral
\[
\int_0^\infty \lambda\hat{f}(\lambda)\hat{f}(-\lambda)d\lambda.
\]

The same argument gives
\begin{equation*}
\begin{aligned}
&\normHT{1/2}{f(2n\cdot)}^2\to \normHR{1/2}{f}^2,\\
&\frac{1}{n}\normHT{1}{f(2n\cdot)}^2\to 2\normHR{1}{f}^2,
\end{aligned}
\quad\text{ as }n\to\infty,
\end{equation*}.
which implies the convergence of the right-hand side of inequality \eqref{res_eq:cbe_est} to the right-hand side of \eqref{res_eq:sb_th}.

The proof of the inequality
\[
\E_\beta e^{\overline{S}_f} \le e^{\frac{1}{\beta}\normHR{1/2}{f}}
\]
for a real-valued $f$ follows similarly from the subgaussian estimate in Corollary \ref{res:cbe_th}.
\end{proof}

\begin{proof}[Proof of Theorem \ref{res:sb_th}]
The desired claim follows by extending the inequality \eqref{res_eq:sb_th} by continuity. According to Lemma \ref{sb_CLT:preliminary}, the inequality holds for smooth compactly supported functions, and these are dense in $H^1(\R)$.

The continuity of all of the expressions entering \eqref{res_eq:sb_th} is straightforward, with the exception of the expectation
$\E_\beta e^{\overline{S}_f}$. We defer the analysis of this expectation (Lemma \ref{sb_CLT:reg_cont}) and the definition of $\overline{S}_f$ for general $f\in H^{1/2}(\R)$ (Definition \ref{def:reg}) to the next subsection. In particular, Lemma \ref{sb_CLT:reg_cont} completes the proof.

The proof of the subgaussian estimate is similar.
\end{proof}

\subsection{Regularization of additive functionals}\label{ssec:regularization}
For a general point process, only additive functionals defined by compactly supported bounded functions are guaranteed to be almost surely finite. In many cases, however, the class of admissible functions may be extended. For example, if the process is translationally invariant, an additive functional may be assigned to any absolutely integrable function, as the first correlation function is constant.

One may attempt to go even further and define the sum
\[
\sum_{x\in X}\frac{1}{i+x}, \quad X\in \Conf(\R),
\]
which, evidently, diverges almsot surely for translationally invariant process. Nevertheless, it is possible for the series to converge in an improper sense:
\[
\lim_{n\to\infty}\sum_{x\in X, \abs{x}\le k_n}\frac{1}{i+x}, \quad X\in \Conf(\R),
\]
for some sequence $k_n\to \infty$ as $n\to\infty$. In other words, while the additive functional $S_{1/(x+i)}$ is not well-defined as an absolutely convergent sum, the limit of the truncated functionals $S_{1/(x+i)\I_{[-k_n, k_n]}}$ may still exist as $n\to\infty$. Such regularization of additive functionals was introduced by Bufetov \cite{B_18} and then applied to the construction of Radon-Nikodym derivatives between Palm measures of determinantal point process.

\begin{proposition}\label{sb_CLT:regularization}
There exists a constant $C$ such that for a smooth compactly supported function $f$ we have
\[
\E_\beta \abs{\overline{S}_f}^2 \le C\normHR{1/2}{f}^2.
\]
\end{proposition}
\begin{proof}
It is sufficient to assume that $f$ is real-valued. A standard calculation implies the desired claim from the subgaussian estimate established in Lemma \ref{sb_CLT:preliminary}.
\end{proof}
Proposition \ref{sb_CLT:regularization} justifies the following definition.
\begin{definition}\label{def:reg}
For a function $f\in H^{1/2}(\R)$, define the respective regularized additive functional $\overline{S}_f$ as the $L^2(\P_\beta, \Conf(\R))$ limit
\[
\overline{S_f} = \lim_{k\to\infty}\overline{S_{f_k}},
\]
for some sequence $f_k$ of compactly supported smooth functions, approximating $f$ in $H^{1/2}(\R)$ norm.
\end{definition}
\begin{remark}
We note that the result of Proposition \ref{sb_CLT:regularization} does not involve $L^2(\R)$ norm. Therefore, one may extend this definition to functions $f$ such that $\normHR{1/2}{f}<+\infty$.
\end{remark}

The following statement concludes the proof of Theorem \ref{res:sb_th}.
\begin{lemma}\label{sb_CLT:reg_cont}
The expectation $\E_\beta e^{\overline{S}_f}$ is continuous with respect to the $1/2$-Sobolev seminorm~$\normHR{1/2}{f}$.
\end{lemma}
\begin{proof}
Indeed, by the Cauchy-Bunyakovsky-Schwarz inequality and Proposition \ref{sb_CLT:regularization} we have
\[
\abs*{\E_\beta e^{\overline{S}_f}-\E_\beta e^{\overline{S}_g}} \le \sqrt{\E_\beta e^{2S_{\Re f+\Re g}}}\sqrt{\E_\beta \abs{\overline{S}_{f-g}}^2} \le \sqrt{C}e^{2\normHR{1/2}{\Re f+\Re g}^2}\normHR{1/2}{f-g}.
\]
\end{proof}

\subsection{Proof of Theorem \ref{res:opt_conv}}\label{ssec:opt_proof}

We approximate the function $f$ by a family of compactly supported smooth functions $f_\varepsilon$ in the $1/2$-Sobolev seminorm. This approximation is uniform in the scaling parameter $R$, as the seminorm is invariant under this scaling:
\[
\normHR{1/2}{f(\cdot/R)-f_\varepsilon(\cdot/R)} = \normHR{1/2}{f(\cdot)-f_\varepsilon(\cdot)}.
\]
The approximating functions $f_\varepsilon$ are chosen so that, for all $R$, the following two inequalities are satisfied:
\begin{align*}
&\abs{\E_\beta e^{\overline{S}_{f(\cdot/R)}}- \E_\beta e^{\overline{S}_{f_\varepsilon(\cdot/R)}}} <\varepsilon/3,\\
&\abs*{e^{\frac{2}{\beta}\int_0^\infty\lambda\hat{f}(\lambda)\hat{f}(-\lambda)d\lambda} - e^{\frac{2}{\beta}\int_0^\infty\lambda\hat{f}_\varepsilon(\lambda)\hat{f}_\varepsilon(-\lambda)d\lambda}} <\varepsilon/3.
\end{align*}
The existence of such an approximation is guaranteed by Lemma \ref{sb_CLT:reg_cont}.

We apply Theorem \ref{res:sb_th} to the smooth, compactly supported function $f_\varepsilon$. This theorem implies that for all sufficiently large $R$,
\[
\abs*{\E_\beta e^{\overline{S}_{f_\varepsilon(\cdot/R)}} -e^{\frac{2}{\beta}\int_0^\infty\lambda\hat{f}_\varepsilon(\lambda)\hat{f}_\varepsilon(-\lambda)d\lambda}} < \varepsilon/3.
\]

Finally, by combining the three estimates above via the triangle inequality, we arrive at the desired result: for all sufficiently large $R$,
\[
\abs*{\E_\beta e^{\overline{S}_{f(\cdot/R)}}- e^{\frac{2}{\beta}\int_0^\infty\lambda\hat{f}(\lambda)\hat{f}(-\lambda)d\lambda} }<\varepsilon.
\]

\section{Proof of Proposition \ref{outline:convergence}}\label{sec:sine_def}

The space of configurations $\Conf(\R)$ is a set of locally finite subsets of the real line with multiplicities. We endow it with the vague topology, i. e., the coarsest topology with respect to which the maps
\[
\Conf(\R)\to \C, \quad X\mapsto S_g(X) = \sum_{x\in X}g(x)
\]
are continuous for all continuous compactly supported functions $g$ on the real line. Equipped with the vague topology, the space $\Conf(\R)$ is Polish. The vague topology is induced by the metric
\begin{equation}\label{eq:pp_metric}
d(X, Y) = \sum_{j\in \N}2^{-j}\frac{\abs{S_{g_j}(X)-S_{g_j}(Y)}}{1+\abs{S_{g_j}(X)-S_{g_j}(Y)}},
\end{equation}
where $\{g_j\}_{j\in\N}$ is a dense countable subset in the space of continuous compactly supported functions. A point process is a Borel probability measure on the space of configurations. Proposition \ref{outline:convergence} follows from the following general statement.
\begin{proposition}\label{unb_conv}
Let $\{\P_n\}_{n\in \N}$ be a tight sequence of point processes. Assume that the sequence converges to a point process $\P$ in the sense of Laplace functionals, i.e.,
\[
\E_n \left\{\prod_{x\in X}e^{-f(x)}\right\} \to \E \left\{\prod_{x\in X}e^{-f(x)}\right\}, \quad \text{ as }n\to\infty
\]
for any smooth, compactly supported, non-negative function $f$. Then these point processes converge to $\P$ weakly.

Assume, furthermore, that $G$ is a continuous (not necessarily bounded) function on $\Conf(\R)$ satisfying
\begin{equation}\label{eq:unb_conv}
\E_n \abs{G} \le C, \quad \E_n \abs{G}(\log_+\abs{G})^{1+\varepsilon}\le C.
\end{equation}
Then $G\in L_1(\Conf(\R), \P)$ and
\[
\E_n G \to \E G.
\]
\end{proposition}

We proceed to the proof of Proposition \ref{outline:convergence}.

Denote by $\tilde\P_\beta^n$ the image of the circular $\beta$-ensemble under the dilation $\theta\mapsto n\theta$ so that the equality
\[
\E_\beta^n \prod_{j=1}^n e^{f(n\theta_j)}=\tilde\E_\beta^n \prod_{j=1}^n e^{f(\theta_j)}
\]
holds. The convergence of $\tilde\P_\beta^n$ in the sense of Laplace functionals was established by Killip and Stoiciu \cite[Theorem 1.7, Definition 1.6]{KS_09}. We substitute the function
\[
G(X) = \prod_{x\in X}e^{f(x)}, \quad X\in\Conf(\R)
\]
into Proposition \ref{unb_conv}. This implies the desired convergence once the requirements \eqref{eq:unb_conv} of Proposition \ref{unb_conv} are verified.

Let us show that the condition \eqref{eq:unb_conv} is satisfied. We may bound $\Re f$ by a smooth positive compactly supported function $g$. We have the respective bound for the multiplicative functionals
\[
\abs{G(X)} \le \prod_{x\in X}e^{g(x)}.
\]
The expectation of the right-hand side under $\tilde\P_\beta^n$ is bounded by the second claim of Corollary \ref{res:cbe_th} (see also proof of Lemma \ref{sb_CLT:preliminary}).

To prove the second inequality in \eqref{eq:unb_conv}, we also require that $\abs{f}\le \abs{g}$ to get the bound
\[
\abs{G(X)}(\log_+(\abs{G(x)}))^{1+\varepsilon} \le \left(\sum_{x\in X}g(x)\right)^{1+\varepsilon}\prod_{x\in X}e^{g(x)}.
\]
We set $\varepsilon=1$ and use that
\[
\tilde \E_\beta^n \left(\sum_{j=1}^n g(\theta_j)\right)^2\prod_{j=1}^n e^{g(\theta_j)} = \left(\frac{d}{d\lambda}\right)^2\bigg |_{\lambda=0}\tilde\E_\beta^n\prod_{j=1}^n e^{\lambda g(\theta_j)}.
\]
The function on the right-hand side is holomorphic in $\lambda$ and the Cauchy formula may be used. Observe that the second part of Corollary \ref{res:cbe_th} gives a bound on the expectation on the right-hand side that is locally uniform in $\lambda$. This implies the boundedness of the left-hand side. We conclude that the condition \eqref{eq:unb_conv} holds.

It is remaining to show that $\{\tilde\P_\beta^n\}_{n\in\N}$ are tight. For a configuration $X$ and a subset $A\subset \R$, set $$\#_A(X)=\abs{A\cap X}.$$ The desired tightness follows from
\begin{proposition}\label{pp:tightness}
A family of point processes $\{\P_\alpha\}$ is tight if so are the distributions of the induced random variables $\#_{[-R, R]}$ for any fixed $R>0$.
\end{proposition}
By Markov's inequality, tightness of the random variables $\#_{[-R, R]}$ follows from their uniform integrability. Their expectation under a circular $\beta$-ensemble follows immediately from the translational invariance:
\[
\tilde\E_\beta^n\#_{[-R, R]} = \frac{R}{\pi}, \quad \text{ for any }n\ge \frac{R}{\pi}.
\]
By Proposition \ref{pp:tightness}, this implies tightness of $\{\tilde\P_\beta^n\}_{n\in\N}$. Proposition \ref{outline:convergence} is proved.

We proceed to the proofs of Propositions \ref{unb_conv}, \ref{pp:tightness}.

\begin{proof}[Proof of Proposition \ref{unb_conv}]
Since weak convergence is metrizable, it suffices to identify the limit of every weakly convergent subsequence.
Assume some subsequence converges weakly to a different measure $\tilde \P$. Weak convergence implies convergence of Laplace functionals. Therefore, these coincide for both measures:
\[
 \E \left\{\prod_{x\in X}e^{-f(x)}\right\} =  \tilde\E \left\{\prod_{x\in X}e^{-f(x)}\right\}, \quad f\in C_c^\infty(\R), f\ge 0.
\]
However, Laplace functionals uniquely determine the law of a point process. We conclude that $\P=\tilde\P$. This proves the weak convergence.

We proceed to the second claim. For $\delta>0$, we introduce a regularization 
\[
G_\delta = G e^{-\delta(\log_+\abs{G})^{1+\varepsilon}}.
\]
The function $G_\delta$ is bounded and continuous. The first claim implies the bound
\[
\E\abs{G_\delta} = \lim_{n\to\infty}\E_n\abs{G_\delta} \le \lim_{n\to\infty}\E_n\abs{G} \le C.
\]
By the Monotone Convergence Theorem,
\[
C\ge \E\abs{G_\delta} \to \E\abs{G}, \quad{ as }\delta\to 0.
\]
Thus making $\delta$ approach zero proves that $G\in L_1(\Conf(\R), \P)$.

Last, we establish the desired convergence. The required bounds \eqref{eq:unb_conv} imply that our regularization is uniform in $n$ in the sense that
\[
\E_n\abs{G-G_\delta} \le \delta\E_n\abs{G}(\log_+\abs{G})^{1+\varepsilon} \le \delta C.
\]
Further, we have that the convergence
\[
\E G_\delta \to \E G
\]
holds as $\delta\to 0$ by the Monotone Convergence Theorem, since $\abs{G_\delta}\le \abs{G}$. Thus the difference
\[
\abs{\E_nG -\E G} \le \delta C +\abs{\E_n G_\delta -\E G_\delta} +\abs{\E G_\delta-\E G}
\]
may be made arbitrarily small for all sufficiently large $n$ by choosing sufficiently small $\delta$.
\end{proof}

Before proving Proposition \ref{pp:tightness}, we identify compacts in the space of configurations.

\begin{proposition}\label{pp:compact_crit}
Assume that the functions $\#_{[-R, R]}(-)$ are bounded on a subset $K\subset \Conf(\R)$ for any $R>0$. Then the subset $K$ is precompact.
\end{proposition}
\begin{proof}
Let the space $\C^\N$ be endowed with the product topology. Recall that the latter is induced by the metric
\[
d(\xi, \eta) = \sum_{j\in \N}2^{-j}\frac{\abs{\xi_j-\eta_j}}{1+\abs{\xi_j-\eta_j}},
\]
where $\xi_j$, $\eta_j$ are coordinates of the vectors $\xi, \eta \in \C^\N$. The map
\[
\Conf(\R) \to \C^{\N}, \quad X\mapsto \{S_{g_j}(X)\}_{j\in\N}
\]
is thus an isometric embedding by the formula \eqref{eq:pp_metric}. The image of the set $K$ is precompact by the Tychonoff Theorem. Indeed, each coordinate in the image does not exceed the maximum of the respective function times the number of particles in its support. The latter is bounded by the requirement of the proposition. Last, the preimage of a precompact subset under an isometric embedding is a precompact set.
\end{proof}

\begin{proof}[Proof of Proposition \ref{pp:tightness}]
Fix $\varepsilon>0$. For $k\in \N$, choose a number $N_k$ such that $$\P_\alpha(\#_{[-k, k]}> N_k)<2^{-k-1}\varepsilon$$ for any $\alpha$. This is possible by the Prokhorov Theorem. Introduce the set
\[
K_\varepsilon = \underset{j\in\N}{\bigcap}\{X\in\Conf(\R): \#_{[-k, k]}(X)\le N_k\},
\]
which is compact by Proposition \ref{pp:compact_crit}. The probability of its complement is
\[
\P_\alpha(K_\varepsilon^c) \le \sum_{j\in\N}\P_\alpha(\#_{[-k, k]}> N_k) \le \varepsilon.
\]
This completes the proof.
\end{proof}

\section{Proof of Corollary \ref{res:KS_conv}}\label{sec:KS_conv}

Corollary \ref{res:KS_conv} is a direct consequence of Theorem \ref{res:cbe_th} and the classical Feller smoothing estimate, which we recall here for convenience.

\begin{theorem}[{\cite[p. 538]{F_66}}]\label{last:feller_est}
Let $F_1$ and $F_2$ be the distribution functions of centered, real-valued random variables, with corresponding characteristic functions $\varphi_1$ and $\varphi_2$. Suppose $F_2$ is continuously differentiable and its density is bounded, with
\[
m = \sup_x\abs{F'(x)}<\infty.
\]
Then, for any $T > 0$, the following bound on the Kolmogorov-Smirnov distance holds:
\[
\sup_{x\in \R}\abs{F_1(x)-F_2(x)} \le \frac{24m}{\pi T} +\frac{1}{\pi}\int_{-T}^T\abs*{\frac{\varphi_1(y)-\varphi_2(y)}{y}}dy.
\]
\end{theorem}
\begin{proof}[Proof of Theorem \ref{res:KS_conv}]
We bound the Kolmogorov-Smirnov distance between the distribution of the random variable $\overline{S}_{f(\cdot/R)}$ and a standard normal distribution. To apply Theorem \ref{last:feller_est}, we identify $F_1$ with the distribution function $F_R$ of $\overline{S}_{f(\cdot/R)}$ and $F_2$ with the standard normal distribution function $F_\mathcal{N}$. Let $\varphi_R$, $\varphi_\mathcal{N}$ be the characteristic functions of the former and the latter respectively.

Theorem \ref{res:cbe_th} provides the following estimate: there exists a constant $K>0$ such that
\[
\abs*{\frac{\varphi_R(k)-\varphi_{\mathcal{N}}(k)}{k}} \le \frac{K}{R}e^{K\abs{k}^2}\abs{k},
\]
Note that the factor $1/R$ arises from the scaling properties of the seminorm, specifically $$\normHR{1}{f(\cdot/R)} = R^{-1/2} \normHR{1}{f(\cdot)}.$$

Using this estimate, we deduce the bound on the integral term in Feller's inequality
\[
\int_{-T}^T\abs*{\frac{\varphi_R(k)-\varphi_{\mathcal{N}}(k)}{k}}dk \le \frac{2KT^2}{R}e^{KT^2}.
\]
Therefore, the full bound is
\[
\sup_x\abs{F_R(x)-F_\mathcal{N}(x)} \le \frac{C_1}{R}e^{KT^2} + \frac{C_2}{T}.
\]
The statement of Corollary \ref{res:KS_conv} follows by choosing $T$ to balance these two terms.
\end{proof}
\section{Concluding remarks}
\subsection{Jack measures}\label{SS_intro:jack}
The simplest yet interesting example of a Jack measure is the Jack-Plancherel measure on partitions of a fixed size $n$ (see \cite[Formula (2.2)]{M_08}). For $\alpha>0$, the probability of a partition $\lambda$ is given by
\[
\mathcal{M}_{PL, n}^{\alpha}(\lambda) = \frac{\abs{\lambda}!\alpha^{\abs{\lambda}}}{H(\lambda, \alpha)H'(\lambda, \alpha)},
\]
where
\[
H(\lambda, \alpha) = \prod_{(i, j)\in \lambda}((\lambda_i-j)+(\lambda'_j-i)\alpha +1), \quad H'(\lambda, \alpha) = \prod_{(i, j)\in \lambda}((\lambda_i-j)+(\lambda'_j-i)\alpha +\alpha),\quad \theta >0.
\]
For $\alpha=1$, this coincides with the Plancherel measure on the set of irreducible representations (indexed by partitions) of the symmetric group. For general $\alpha>0$, it can be obtained as a degeneration of the $z$-measures introduced by Borodin and Olshanski \cite{BO_05} (the limit $z, z'\to\infty$ in \cite[(1.1)]{BO_05}).

A remarkable property of the Plancherel measures is the determinantal structure of their Poissonizations
\[
\mathcal{M}_{PL, q}^\alpha(\lambda) = e^{-q^2}\frac{q^{2\abs{\lambda}}}{\abs{\lambda}!}\mathcal{M}_{PL, \abs{\lambda}}^\alpha(\lambda), \quad q>0
\]
under the embedding
\begin{equation}\label{intro_eq:diagram_emb_s}
\Y\to\Conf(\Z),\quad \lambda\mapsto \{\lambda_j-j\}_{j\in\N}\in\Conf(\Z).
\end{equation}
The observation that a random Young diagram, after a 45-degree rotation, yields a determinantal point process is due to Okounkov \cite{O_01}. His result actually applies to a larger class of measures, which we now introduce.

For arbitrary specializations $\rho_1, \rho_2$ such that the series
\[
Z = \sum_{\lambda\in\Y}s_\lambda(\rho_1)s_\lambda(\rho_2)
\]
converges absolutely, the Schur measure is defined by
\[
\mathcal{M}^1_{\rho_1, \rho_2}(\lambda) = Z^{-1}s_\lambda(\rho_1)s_\lambda(\rho_2).
\]
The measure defined this way may, in general, be complex-valued. It is a probability measure if, for example, $\rho_1$ is the complex conjugate of $\rho_2$, or if both specializations are Schur-positive (the latter are classified by the Edrei-Thoma Theorem, see \cite[Corollary 4.2, Proposition 4.4]{BO_17}).

Returning to Okounkov's work \cite{O_01}, it is shown that a Schur measure is determinantal for arbitrary specialization \cite[Theorem 1]{O_01}, \cite[Section 3]{J_01} under the embedding \eqref{intro_eq:diagram_emb_s}. To recover the Plancherel measures, one takes the specialization
\begin{equation}\label{intro_eq:pl_spec}
p_1(\rho_1)=p_1(\rho_2)=q, \quad p_j(\rho_1)=p_j(\rho_2)=0, \text{ for }j\ge 2.
\end{equation}

Following Moll \cite[Definition 1.2.1]{M_23} and Dimitrov, Gao, Gu, Niedernhofer \cite[Definition 2.5]{DGGN_25}, we define the Jack measure for specializations $\rho_1$, $\rho_2$ as
\begin{equation}\label{intro_eq:jack_def}
\mathcal{M}^\alpha_{\rho_1, \rho_2}(\lambda) = Z^{-1}\frac{J^{(\alpha)}_\lambda(\rho_1)J^{(\alpha)}_\lambda(\rho_2)}{\alsc{\J{\lambda}, \J{\lambda}}}, \quad Z = \sum_{\lambda\in\Y}\frac{J^{(\alpha)}_\lambda(\rho_1)J^{(\alpha)}_\lambda(\rho_2)}{\alsc{\J{\lambda}, \J{\lambda}}}.
\end{equation}
For $\alpha=1$, this recovers Schur measures. For general $\alpha>0$, Poissonized Jack-Plancherel measure $\mathcal{M}_{PL, q}^{\alpha}$ arises from the Jack measures through the same specialization \eqref{intro_eq:pl_spec} (see \cite[Formula (10.29)]{M_79} and \cite[Formula (4.2)]{JM_15}).

For a general $\alpha$, we consider the embedding
\begin{equation}\label{intro_eq:diagram_emb}
\Y\to\Conf(\R), \quad \lambda\mapsto \{\lambda_j-\alpha j\}.
\end{equation}
Strahov \cite{S_08, S_10} showed that Jack-Plancherel measures and, more generally,$z$-measures for parameters $\alpha=2, 1/2$ are Pfaffian point processes under the embedding \eqref{intro_eq:diagram_emb}. Recall that for the corresponding parameters $\beta=1, 4$, the circular $\beta$-ensemble is also a Pfaffian point process. This observation motivates the following question, inspired by Theorem \ref{intro:gessel}.

\begin{question}
For which specializations $\rho_1, \rho_2$ are the Jack measures corresponding to $\alpha=2, 1/2$ Pfaffian point processes under the embedding \eqref{intro_eq:diagram_emb}?
\end{question}

Another interesting question concerns the behavior of Jack measures under the microscopic scaling limit for the respective specialization. For a smooth, compactly supported, real-valued function $f$, consider the specializations
\[
p_j(\rho_1^n) = \overline{p_j(\rho_2^n)} = \alpha j\widehat{f(n\cdot)}_j,
\]
where, for sufficiently large $n$, the function $f(n\theta)$ can be regarded as a function on the unit circle. We rescale the embedding \eqref{intro_eq:diagram_emb} as follows
\begin{equation}\label{intro_eq:diagram_scaling}
\{\lambda_j-\alpha j\}_{j\in\N}\mapsto \left\{\frac{\lambda_j-\alpha j}{n}\right\}_{j\in\N}.
\end{equation}
\begin{question}\label{quest:jack_conv}
Under the scaling \eqref{intro_eq:diagram_scaling}, do the Jack measures $\mathcal{M}^{\alpha}_{\rho_1^n, \rho_2^n}$ converge to a limiting point process?
\end{question}

In the case $\alpha=1$, corresponding to the Schur measures, the answer is affirmative. This follows from the calculations by Bufetov \cite{B_24}. The limiting process is determinantal; its kernel is given by a product of two Hankel operators. The Fredholm determinant of this kernel describes the convergence of additive functionals under the sine process.

There is a substantial body of work on the limiting behavior of $z$-measures and Jack measures. Borodin and Olshanski \cite{BO_05} studied the limit of $z$-measures on the boundary of the Young graph. The asymptotic shape of a large partition under the Jack-Plancherel measure, as well as Gaussian fluctuations, were investigated by Dolega and F\'eray \cite{DF_16, DS_19}. Matsumoto \cite{M_08} established a connection between large partitions under the Jack measure and the traceless Gaussian $\beta$-ensemble. Moll \cite{M_23} derived the limit shape of a large partition under the Jack measure for general specializations. Finally, the asymptotic behavior of Jack measures with varying homogeneous specializations was considered by Dimitrov, Gao, Gu, and Niedernhofer \cite{DGGN_25}.

\subsection{From the fundamental equation of Johansson to the Dyson Brownian motion and Jack polynomials}
This section explains the connection between the Jack-polynomial approach and a fundamental equation introduced by Johansson. We begin by presenting the equation exactly as it appears in \cite[Lemma 2.1]{L_21}.
\begin{proposition}
For any functions $w\in C^1(\mathbb{T})$ and $g\in C^1(\mathbb{T})$, the following identity holds:
\[
\E_\beta^n\left\{\prod_{j=1}^n e^{w(\theta_j)}\left[\frac{\beta}{2}\sum_{p, q=1}^n\frac{g(\theta_p)-g(\theta_q)}{2\tan\left(\frac{\theta_p-\theta_q}{2}\right)}+\left(1-\frac{\beta}{2}\right)\sum_{j=1}^ng'(\theta_j)+\sum_{j=1}^ng(\theta_j)w'(\theta_j)\right]\right\}=0.
\]
\end{proposition}
This identity naturally leads to the introduction of a certain differential operator. Substituting $g = w'$ and noting that the diagonal terms of the first sum contribute exactly the second sum, one obtains
\[
\E_\beta^n\mathcal{H}_\beta^n\prod_{j=1}^n e^{w(\theta_j)},
\]
where $\mathcal{H}_\beta^n$ acts on symmetric functions on $\mathbb{T}^n$ as
\[
\mathcal{H}_\beta^n = \sum_{j=1}^n\partial_{\theta_j}^2 + \frac{\beta}{2}\sum_{1\le p< q\le n}\frac{\partial_{\theta_p}-\partial_{\theta_q}}{\tan\left(\frac{\theta_p-\theta_q}{2}\right)}.
\]
In particular, the circular $\beta$-ensemble is invariant under the dynamics generated by $\mathcal{H}_\beta^n$. The operator $\mathcal{H}_\beta^n$ is the generator of the circular \textit{Dyson Brownian motion} \cite[Definition 9, Remark 11]{W_16}:
\[
d\theta_j(t) = \sqrt{2}dW_j(t) + \frac{\beta}{2}\sum_{k\ne j}\cot\left(\frac{\theta_j-\theta_k}{2}\right)dt, \quad j=1,\ldots,n.
\]
Moreover, this Brownian motion is reversible: the operator $\mathcal{H}_\beta^n$ is self-adjoint in $L^2(\P_\beta^n)$ (with an appropriate domain), i.e.,
\[
\E_\beta^n\mathcal{H}_\beta^nh_1\overline{h_2}=\E_\beta^nh_1\mathcal{H}_\beta^n\overline{h_2}.
\]
Consequently, its eigenfunctions are orthogonal with respect to the circular $\beta$-ensemble.

In spite of Proposition \ref{Jack:orth}, it is natural to expect that the Jack polynomials provide the desired eigenfunctions. Rewriting $\mathcal{H}_\beta^n$ in the variables $z_j = e^{i\theta_j}$ gives
\[
\mathcal{H}_\beta^n = -\sum_{j=1}^n(z_j\partial_{z_j})^2 - \frac{\beta}{2}\sum_{1\le p<q\le n}\frac{z_p+z_q}{z_p-z_q}(z_p\partial_{z_p}-z_q\partial_{z_q}),
\]
so that finding eigenfunctions reduces to an algebraic problem: the operator preserves the degree of symmetric polynomials. The precise result is due to Stanley \cite[Theorem 3.1]{S_89}.

\appendix
\section{Definition of the sine-$\beta$ process}\label{a:sine_def}
The goal of this section is to recall the definition of the sine-$\beta$ process.
Following \cite{KS_09}, it can be constructed as the preimage of a randomly shifted
lattice $\psi^{-1}(\omega + 2\pi\mathbb{Z})$ under a random monotone increasing function
$\psi$. We explain below that this construction indeed produces a Borel probability
measure on the space of configurations.

\subsection{Construction of point processes via random monotone increasing functions}
Recall that the space of configurations $\Conf(\mathbb{R})$ consists of locally finite
subsets of $\mathbb{R}$ with multiplicities. We equip it with the vague topology
(see Section~\ref{sec:sine_def}). For a configuration $X$ and a Borel set
$A \subset \mathbb{R}$, we write
\[
\#_A(X) = |A \cap X|.
\]

The Borel $\sigma$-algebra on $\Conf(\mathbb{R})$ is generated by the cylinder sets
\[
C_{k,a,b} = \{X \in \Conf(\mathbb{R}) : \#_{(a,b)}(X) = k\},
\qquad a,b \in \mathbb{R},\; a<b,\; k\in\mathbb{Z}_{\ge0}.
\]

Let $M(\mathbb{R})$ denote the space of monotone increasing functions on
$\mathbb{R}$, endowed with the cylinder $\sigma$-algebra. For $u\in M(\mathbb{R})$
let $c(u) \subset \mathbb{R}$ denote the set of its critical values, i.e. values
with more than one preimage. For a monotone function this set is at most countable.

Define
\[
M_{\mathrm{reg}}
=
\{(u,\omega)\in M(\mathbb{R})\times[0,2\pi]:
c(u)\cap(2\pi\mathbb{Z}+\omega)=\varnothing\},
\]
and equip this subset with the product $\sigma$-algebra.
For a probability measure $\mathcal{P}$ on $M(\mathbb{R})$ define the product measure
\[
\tilde{\mathcal P}=\mathcal{P}\otimes \frac{1}{2\pi}\mathrm{Leb}
\]
on $M(\mathbb{R})\times[0,2\pi]$.
\begin{proposition}
For any probability measure $\mathcal{P}$ on $M(\mathbb{R})$, the set
$M_{\mathrm{reg}}$ has full $\tilde{\mathcal P}$-measure.
\end{proposition}
\begin{proof}
We first show that $M_{\mathrm{reg}}$ is measurable. Write
\[
M_{\mathrm{reg}}=\bigcap_{k\in\mathbb{Z}} A_k,
\qquad
A_k=\{(u,\omega):\omega+2\pi k\notin c(u)\}.
\]

Each set $A_k$ is measurable. Indeed, if $x$ is a critical value of $u$,
then there exist at least two rational points for which $u$ attains the value $x$, since $u$ is monotone. Hence
\[
A_k^c
=
\bigcup_{q_1,q_2\in\mathbb{Q}}
B_{q_1,q_2},
\qquad
B_{q_1,q_2}
=
\{(u,\omega):u(q_1)=u(q_2)=\omega+2\pi k\}.
\]

Each set $B_{q_1,q_2}$ is the preimage of $(0,0)$ under the map
\[
(u,\omega)
\mapsto
\big(u(q_1)-\omega-2\pi k,\;u(q_2)-\omega-2\pi k\big),
\]
which is continuous in the product topology and therefore measurable.
Thus $M_{\mathrm{reg}}$ is measurable.

To show that $M_{\mathrm{reg}}$ has full measure, compute
\[
\tilde{\mathcal P}(M_{\mathrm{reg}}^c)
=
\frac{1}{2\pi}
\int_{M(\mathbb{R})}
\mathrm{Leb}\big(
\{\omega\in[0,2\pi]:(2\pi\mathbb{Z}+\omega)\cap c(u)\neq\varnothing\}
\big)
\,d\mathcal{P}(u).
\]
Since $c(u)$ is countable, the set inside the Lebesgue measure is countable,
and therefore has zero measure. This completes the proof.
\end{proof}

Define the map
\[
\Theta:M_{\mathrm{reg}}\to\Conf(\mathbb{R}),
\qquad
(u,\omega)\mapsto u^{-1}(\omega+2\pi\mathbb{Z}),
\]
where a point $x$ is included in $u^{-1}(y)$ if the function $u$
jumps through the value $y$, i.e.
\[
u_-(x)\le y\le u_+(x).
\]

\begin{proposition}
The map
\[
\Theta:M_{\mathrm{reg}}
\to
\Conf(\mathbb R)
\]
is measurable.
\end{proposition}

\begin{proof}
It suffices to show that, for any interval $(a,b)$, the function
\[
(u, \omega)\mapsto \abs{(a, b)\cap u^{-1}(\omega + 2\pi \Z)}
\]
is measurable.
This follows once we verify that each $u^{-1}(\omega+2\pi k)$ is measurable as a function of $(u, \omega)$.

The desired measurability follows from the union representation
\[
\{u^{-1}(\omega+2\pi k)\in(a,b)\}
=
\bigcup_{n\in\mathbb{N}}
\{(u,\omega):u(a+1/n)\le\omega+2\pi k\le u(b-1/n)\},
\]
where each set in this union is measurable. The claim follows.
\end{proof}

\begin{definition}\label{KSC:pp_def}
Let $\mathcal{P}$ be a probability measure on $M(\mathbb{R})$.
The associated point process is defined as the pushforward
\[
\P_{\mathcal{P}} := \Theta_*\tilde{\mathcal{P}}.
\]
In particular, finite-dimensional distributions of a stochastic process
$X_x$ define a point process whenever $X_a\le X_b$ almost surely for
$a\le b$.
\end{definition}
\subsection{Definition of the sine-$\beta$ process}
Let $B_1$ and $B_2$ be independent standard Brownian motions.
For $x\in\mathbb{R}$, consider the stochastic differential equation
\begin{equation}\label{KSC_eq:sde_def}
dX_x^\beta(t)
=
x\,dt
+
\frac{2}{\sqrt{\beta t}}
\Im\Big\{
\big(e^{iX_x^\beta(t)}-1\big)
\big(dB_1(t)+i\,dB_2(t)\big)
\Big\}.
\end{equation}
\begin{theorem}[{\cite[Proposition 4.5]{KS_09}}]\label{KSC:SDE_sol}
For each $x\in\mathbb{R}$ there exists a unique strong solution
$X_x^\beta$ on $(0,\infty)$ satisfying

\begin{itemize}
\item If $x>y$ then $X_x^\beta(t)-X_y^\beta(t)$ is almost surely
a nonnegative function of $t$.

\item The expectation at $t\in (0, +\infty)$ is
\[
\E X_x^\beta(t)=xt .
\]
\end{itemize}
\end{theorem}
\begin{definition}[{\cite[p.~6]{KS_09}}]\label{KSC:sb_def}
Let $\mathcal{P}_\beta$ denote the finite-dimensional distributions
of the process $X_x^\beta(1)$ as a function of $x$. The \emph{sine-$\beta$ process} is the
point process associated with $\mathcal{P}_\beta$:
\[
\P_\beta:=\P_{\mathcal{P}_\beta},
\]
in the sense of Definition~\ref{KSC:pp_def}.
\end{definition}

\end{document}